\documentclass{article}

\usepackage{amsmath}
\usepackage{amssymb}
\usepackage{amsthm} 
\usepackage[colorlinks=true,urlcolor=blue,linkcolor=blue,plainpages=false,pdfpagelabels]{hyperref}
\usepackage{a4wide}

\theoremstyle{plain}
\newtheorem{theorem}{Theorem}[section]
\newtheorem{corollary}[theorem]{Corollary}
\newtheorem{lemma}[theorem]{Lemma}
\newtheorem{proposition}[theorem]{Proposition}
\theoremstyle{definition}
\newtheorem{definition}[theorem]{Definition}
\newcommand{\bdfn}{\begin{definition}}
\newcommand{\edfn}{\end{definition}}
\newcommand{\bthm}{\begin{theorem}}
\newcommand{\ethm}{\end{theorem}}
\newcommand{\blem}{\begin{lemma}}
\newcommand{\elem}{\end{lemma}}
\newcommand{\bprop}{\begin{proposition}}
\newcommand{\eprop}{\end{proposition}}

\newcommand{\be}{\begin{enumerate}}
\newcommand{\ee}{\end{enumerate}}

\newcommand{\lp}{\left(}
\newcommand{\rp}{\right)}
\newcommand{\lb}{\left[}
\newcommand{\rb}{\right]}
\newcommand{\lbr}{\left\lbrace}
\newcommand{\rbr}{\right\rbrace}

\newcommand{\beq}{\begin{equation}}
\newcommand{\eeq}{\end{equation}}

\newcommand{\R}{{\mathbb R}}
\newcommand{\N}{{\mathbb N}}
\newcommand{\Ns}{\N^*}
\newcommand{\abs}[1]{\left\vert#1\right\vert}
\newcommand{\norm}[1]{\left\lVert#1 \right\rVert}
\newcommand{\ceil}[1]{\left\lceil#1\right\rceil}

\newcommand{\xn}{x_n}
\newcommand{\xnp}{x_{n + 1}}

\newcommand{\Mabd}{M_{\alpha,\beta,\delta}}
\newcommand{\Mr}{M_r}
\newcommand{\Kp}{K_p}
\newcommand{\Kpz}{K_p^{0}}

\newcommand{\remin}{\mathop{-\!\!\!\!\!\hspace*{1mm}\raisebox{0.5mm}{$\cdot$}}\nolimits}

\newcommand{\rateoned}{\sigma_1}

\newcommand{\ratetwod}{\sigma_2}
\newcommand{\ratethreed}{\sigma_3}
\newcommand{\rateonea}{\theta_1}
\newcommand{\rateoneb}{\gamma_1}
\newcommand{\ratetwob}{\gamma_2}
\newcommand{\rateoner}{\lambda_1}
\newcommand{\ratetwor}{\lambda_2}
\newcommand{\rateprod}{A}
\newcommand{\rateproddelta}{\sigma_1^*}

\title{Quantitative asymptotic regularity  and $T$-asymptotic regularity for the inexact  generalized Halpern iteration}
\author{Nicoleta Dumitru$^a$ and Lauren\c tiu Leu\c stean$^{a, b, c}$ \\ [2mm]
\footnotesize{$^a$ LOS, Faculty of Mathematics and Computer Science, University of Bucharest} \\
\footnotesize{$^b$ Simion Stoilow Institute of Mathematics of the Romanian Academy} \\
\footnotesize{$^c$ Institute for Logic and Data Science, Bucharest} \\
\footnotesize{E-mails: \href{mailto:nicoleta.dumitru@my.fmi.unibuc.ro}{nicoleta.dumitru@my.fmi.unibuc.ro}, \href{mailto:laurentiu.leustean@unibuc.ro}{laurentiu.leustean@unibuc.ro}}
}

\date{}

\begin{document}

\date{}
\maketitle

\begin{abstract}
\noindent
We apply proof mining techniques to obtain quantitative and qualitative results on asymptotic and $T$-asymptotic regularity for the inexact generalized Halpern iteration, a viscosity-type extension of an iteration recently studied by Kanzow and Shehu.  Specializing our results to the Kanzow--Shehu iteration and the sequential averaging method (SAM) yields analogous results for these iterations. Furthermore, we compute rates of ($T$-)asymptotic regularity for particular choices of the parameter sequences, and for one of them, we obtain linear rates as an application of a  lemma due to Sabach and Shtern. \\



\end{abstract}

\section{Introduction}

Let $X$ be a Banach space, and let $T : X \to X$ be a nonexpansive mapping with fixed points. Recently, Kanzow and Shehu \cite[Section 4]{KanShe17} 
introduced the following very general iteration: 
\begin{equation}\label{KanShe-iteration}
x_0 \in X, \qquad \xnp = \delta_n u + \alpha_n \xn + \beta_n T \xn + r_n, 
\end{equation}
where $u \in X$, $(r_n) \subseteq X$ is a sequence of residuals, and $(\alpha_n)$, $(\beta_n)$, $(\delta_n)$ are sequences in $[0,1]$  satisfying 
$\alpha_n + \beta_n + \delta_n \leq 1$ for all $n \in \N$. 

In this paper, we study the asymptotic behaviour of the iteration 
\begin{equation}\label{inexact-gen-Halpern}
x_0 \in X, \qquad \xnp = \delta_n f(x_n) + \alpha_n \xn + \beta_n T \xn + r_n,
\end{equation}
which we refer to as the \textit{inexact generalized Halpern iteration}, where $f:X \to X$ is a $\rho$-contraction for some $\rho \in [0,1)$. 
If $f$ is a constant mapping, then $(\xn)$ reduces to the iteration of Kanzow and Shehu \eqref{KanShe-iteration}. Hence
the inexact generalized Halpern iteration can be viewed as a viscosity-type extension (in the sense of \cite{Att96,Mou00}) 
of the Kanzow--Shehu iteration.
The special case obtained by setting $r_n=0$ for all $n\in \N$  in \eqref{inexact-gen-Halpern} 	
will be referred to as the \textit{generalized Halpern iteration}. Moreover, if we set, for all $n \in \N$, $\alpha_n =0$, $\beta_n =1- \delta_n$, and 
$r_n=0$ in \eqref{inexact-gen-Halpern}, we obtain a viscosity version of the Halpern iteration studied by Xu \cite{Xu04} in Banach spaces and later by Sabach and Shtern \cite{SabSht17} in  Hilbert spaces, 
who called it the sequential averaging method (SAM). 

Let us recall that the sequence $(x_n)$ is said 
to be asymptotically regular if 
\[\lim\limits_{n\to \infty} \norm{x_{n+1}-x_n}=0,\] 
and $T$-asymptotically regular 
if 
\[\lim\limits_{n\to \infty} \norm{x_n-Tx_n}=0.\] 
A rate of asymptotic regularity of $(x_n)$ is a rate of convergence to $0$ of  $\lp\norm{x_{n+1}-x_n}\rp$, while a 
rate of $T$-asymptotic regularity of $(x_n)$ is a rate of convergence to $0$ of  $\lp\norm{x_n-Tx_n}\rp$. 
If $\Phi:\N\to \N$ is a rate of ($T$-)asymptotic regularity of $(x_n)$, we say that $(x_n)$ is ($T$-)asymptotically regular with rate $\Phi$. 

Our main results are Theorems \ref{thm-t-as-reg-quant-Q1} and  \ref{thm-t-as-reg-quant-Q1s}, which provide uniform rates of asymptotic regularity and $T$-asymptotic regularity for the inexact generalized Halpern iteration 
$(x_n)$, under suitable assumptions on the parameter sequences $(\alpha_n)$, $(\beta_n)$, $(\delta_n)$, $(r_n)$. These theorems are obtained by extending to 
this iteration the proof mining methods developed in \cite{Leu07a,Koh11,KohLeu12a} for the Halpern iteration. We refer, for example, to \cite{Koh08,Koh19} for further details on 
the program of proof mining. Recently, in \cite{FirLeu25a}, the authors obtained rates of  ($T$-)asymptotic regularity for the generalized Krasnoselskii-Mann-type iteration 
also studied  by Kanzow and Shehu in \cite[Section 3]{KanShe17}. However, the proof mining techniques and the quantitative hypotheses on the parameter sequences from 
\cite{FirLeu25a} differ from those used in the present paper. As an immediate consequence of our quantitative results we obtain Theorem \ref{main-thm-t-as-reg-qualitative}, 
a new asymptotic regularity and $T$-asymptotic regularity theorem for $(x_n)$. By specializing  our results to
the Kanzow--Shehu iteration and the sequential averaging method (SAM), we obtain new quantitative and qualitative ($T$-)asymptotic regularity  results for these iterations as well.

Finally, in Section \ref{section-examples}, we compute rates of ($T$-)asymptotic regularity for two particular choices of 
the parameter sequences satisfying the quantitative hypotheses of our main theorems. 
Furthermore, we give a third example for which linear rates are obtained  as an 
application of a well-known lemma on sequences of real numbers due to Sabach and Shtern \cite{SabSht17}.

\mbox{}

\noindent \textbf{Notation:} We denote $\Ns=\N\setminus \lbr 0 \rbr$. For $x\in \R$, $\lceil x \rceil$ denotes the ceiling of $x$. The set of fixed points of the mapping $T:X\to X$ is denoted by 
$\operatorname{Fix}(T)$. For all $m,n\in\N$, we set $m\remin n = \max\{m-n,0\}$. For every $g:\N\to \N$, the mapping $g^+:\N\to \N$ is defined by $g^+(n)=\max\{g(i)\mid 0\leq i\leq n \}$.

\section{Quantitative notions and lemmas}

Let $(a_n)_{n \in \N}$ be a sequence of real numbers and $\varphi : \N \to \N$. 
If $\lim\limits_{n\to \infty} a_n = a$, then 
$\varphi$ is said to be a rate of convergence of $(a_n)$ to $a$ if 
\[\forall k \in \N \, \forall n \geq \varphi(k)  \lp \abs{a_n - a} \leq \frac{1}{k+1} \rp,\]
and  $\varphi$ is a Cauchy modulus of $(a_n)$ if 
\[\forall k \in \N \; \forall n \geq \varphi(k) \, \forall l \in \N  \lp \abs{a_{n + l} - a_n} \leq \frac{1}{k+1} \rp.\]
Assume moreover that $(a_n) \subseteq [0,\infty)$. Then $\varphi$ is a Cauchy modulus of the series  $\sum\limits_{n=0}^\infty a_n$  if
\[\forall k\in \N\, \forall n\geq \varphi(k)\, \forall l\in\N^* \lp \sum\limits_{i=n+1}^{n+l} a_i\leq \frac1{k+1}\rp.\] 
If  $\sum\limits_{n=0}^\infty a_n$ is divergent, then $\varphi$ is a rate of divergence of $\sum\limits_{n=0}^\infty a_n$ if 
$\sum\limits_{i = 0}^{\varphi(n)} a_i \geq n$ for 
all $n \in \N$. 

The following quantitative lemmas will be used in the paper.

\blem \label{cauchy-mod-up-bd-conv-rate} \cite[Lemma 3.1]{FirLeu25a}\, \\
Let $(a_n)\subseteq [0,\infty)$ be  such that  $\sum\limits_{n=0}^\infty a_n$ converges with 
Cauchy modulus $\varphi$. 
Then 
\be
\item\label{cauchy-mod-up-bd-conv-rate-1} $\sum\limits_{n=0}^\infty a_n \leq M$, where $M \in \Ns$ is such that 
$M \geq \sum\limits_{i = 0}^{\varphi(0)} a_i + 1$.
\item\label{cauchy-mod-up-bd-conv-rate-2} $\lim\limits_{n\to \infty} a_n = 0$ with rate of convergence $\psi(k)=\varphi(k)+1$. 
\ee
\elem

\begin{lemma}\label{series-Cauchy-modulus-linear-comb}\cite[Lemma 2.4(iii)]{FirLeu26}\, \\
Let $(a_n), (b_n) \subseteq [0,\infty)$, $q,r\in (0,\infty)$, and define $c_n=qa_n+rb_n$ for all $n\in \N$.
Assume that $\sum\limits_{n=0}^\infty a_n$ is Cauchy with modulus $\varphi_1$ and 
$\sum\limits_{n=0}^\infty b_n$ is Cauchy with modulus $\varphi_2$.
Then $\sum\limits_{n=0}^\infty c_n$ is Cauchy with modulus $\varphi$ defined as follows:
\[\varphi:\N\to\N, \quad \varphi(k)=\max\left\{\varphi_1(\lceil 2q(k+1)\rceil-1), \varphi_2(\lceil 2r(k+1)\rceil-1)\right\}.\]
\end{lemma}

\begin{lemma}\label{useful-Cauchy-rate}
Let $t\in [0, \infty)$, $L\in \N^*$ and define
\begin{eqnarray*}
& \varphi(k) =\ceil{t(k+1)}, \quad  & \varphi^*(k)= \ceil{t(k+1)} \remin L,  \\
& \psi(k) = \ceil{\sqrt{t(k+1)}}, \quad  & \psi^*(k)=\ceil{\sqrt{t(k+1)}} \remin L.  
\end{eqnarray*}
Then  
\be
\item\label{useful-Cauchy-rate-1} $\varphi$ and $\varphi^*$ are Cauchy moduli of $\sum\limits_{n=0}^{\infty}\frac{t}{(n+L)^2}$.
\item\label{useful-Cauchy-rate-1u} $\sum\limits_{n=0}^{\infty}\frac{t}{(n+L)^2} \leq 2\ceil{t}$.
\item\label{useful-Cauchy-rate-2} $\varphi$ and $\varphi^*$ are rates of convergence of $\lp \frac{t}{n+L}\rp$ to $0$. 
\item\label{useful-Cauchy-rate-3} $\psi$ and $\psi^*$ are rates of convergence of $\lp \frac{t}{(n+L)^2}\rp$ to $0$. 
\ee
\end{lemma}
\begin{proof}
 \eqref{useful-Cauchy-rate-1},  \eqref{useful-Cauchy-rate-1u} and \eqref{useful-Cauchy-rate-2} are obtained by (an adaptation of  the proof of)
 \cite[Lemma 5.1]{FirLeu25a}. Let us prove \eqref{useful-Cauchy-rate-3}. As $\psi(k) \geq \psi^*(k)$ for all $k\in \N$, 
 it is enough to prove that 
 $\psi^*$ is a rate of convergence. Let $k\in\N$ and $n\geq \psi^*(k)\geq \ceil{\sqrt{t(k+1)}}-L$. We get that 
 $\frac{t}{(n+L)^2}  \leq \frac{t}{\lp\ceil{\sqrt{t(k+1)}}\rp^2} \leq \frac1{k+1}$. 
\end{proof}

\begin{proposition}\label{quant-Xu}  \, \\
Let $(s_n), (c_n)\subseteq [0,+\infty)$  and $(a_n)\subseteq [0,1]$ satisfy, for all $n\in\N$,
\[s_{n+1}\leq (1-a_n)s_n + c_n. \]
Assume that $L\in\N^*$ is an upper bound on $(s_n)$ and  $\sum\limits_{n=0}^\infty c_n$ converges 
with Cauchy modulus $\chi$.
\be
\item\label{quant-Xu-1} If $\sum\limits_{n=0}^\infty a_n$  diverges with rate of 
divergence $\theta$, then $\lim\limits_{n\to\infty} s_n=0$ with rate of convergence $\Sigma$ defined by 
\[
\Sigma(k)=\theta \lp \chi(2k+1)+1+ \ceil{\ln(2L(k+1))}\rp +1.
\]
\item\label{quant-Xu-2}  Suppose that $\rateprod:\N \times \N \to \N$ satisfies the following: 
\beq\label{quant-Xu-2-hyp}
\forall m,k \in \N \lp A(m,k)\geq m \text{ and }\prod\limits_{i=m}^{\rateprod(m,k)}(1-a_i)\leq \frac1{k+1}\rp.
\eeq
Then $\lim\limits_{n\to \infty} s_n = 0$ with rate of convergence
\[
\Sigma^*(k) = \rateprod \lp \chi(2k+1)+1, 2L(k+1)-1 \rp + 1.
\]
\ee
\end{proposition}
\begin{proof}
\be
\item By \cite[Proposition 2.7]{LeuPin24}.
\item Take $\varepsilon = \frac1{k+1}$ in \cite[Lemma 2.10.(1)]{DinPin23}.
\ee
\end{proof}

Proposition \ref{quant-Xu} is a quantitative version of Xu's lemma \cite{Xu02}, a very useful tool in the study of the asymptotic behaviour of nonlinear iterations. 
Quantitative versions of Xu's lemma related to the convergence to $0$ of the product $\prod\limits_{n=0}^\infty (1-a_n)=0$ were first applied by Kohlenbach \cite{Koh11}
and by Kohlenbach and the second author \cite{KohLeu12a} in the study of the Halpern iteration. 
A mapping $\rateprod:\N \times \N \to \N$ satisfying \eqref{quant-Xu-2-hyp} was first considered 
by Kohlenbach \cite[Lemma 10]{Koh20a} and later used in \cite{Pin21,KohPin22,DinPin23}.

\section{A generalized Halpern iteration}\label{gemMH-iteration}

In the sequel, $X$ is a normed space, $T:X \to X$ is a nonexpansive mapping with fixed points,  
$f:X \to X$ is a $\rho$-contraction for some $\rho \in [0,1)$, $(r_n)_{n\in \N}$ is a sequence in $X$, and 
$(\alpha_n)_{n\in \N}$, $(\beta_n)_{n\in \N}$, $(\delta_n)_{n\in \N}$ are sequences in $[0,1]$ such that for all $n\in \N$, 
\[\alpha_n + \beta_n + \delta_n \leq 1. \]

Let $(\xn)$ be the inexact generalized Halpern iteration, defined by \eqref{inexact-gen-Halpern}: 
\[ x_0 \in X, \qquad \xnp = \delta_n f(x_n) + \alpha_n \xn + \beta_n T \xn + r_n. \]

\blem
For all $n\in \N$ and $p \in \operatorname{Fix}(T)$,
\begin{align}
\norm{\xnp - p} & \leq \delta_n \norm{f(\xn) - p} + (1 - \delta_n) \norm{\xn - p} + 
(1 - \alpha_n - \beta_n - \delta_n)\norm{p} + \norm{r_n}. \label{xn-p-ineq-viscosity}
\end{align}
\elem
\begin{proof}
Let $n \in \N$ and $p \in \operatorname{Fix}(T)$. We get that 
\begin{align*}
\norm{\xnp - p} & = 
\norm{\delta_n f(\xn) + \alpha_n \xn + \beta_n T \xn + r_n - p} \\
&  = \norm{\delta_n (f(\xn) - p) + \alpha_n (\xn - p) + \beta_n (T \xn - p) + r_n + (\alpha_n + \beta_n + \delta_n - 1) p} \\
& \leq  \delta_n \norm{f(\xn) - p} + \alpha_n \norm{\xn - p} + \beta_n \norm{T \xn - p} + \norm{r_n} + 
(1 - \alpha_n - \beta_n - \delta_n)\norm{p} \\
& \leq \delta_n \norm{f(\xn) - p} + (\alpha_n + \beta_n) \norm{\xn - p} + \norm{r_n} + (1 - \alpha_n - \beta_n - \delta_n)\norm{p}  \\
& \quad \text{since $\norm{T \xn - p}=\norm{T \xn - Tp} \leq \norm{\xn - p}$, as $p \in \operatorname{Fix}(T)$ and $T$ is nonexpasive}\\
& \leq \delta_n \norm{f(\xn) - p} + (1 - \delta_n) \norm{\xn - p} + (1 - \alpha_n - \beta_n - \delta_n)\norm{p} + \norm{r_n}.
\end{align*}
\end{proof}

Let $p\in \operatorname{Fix}(T)$ and define the sequence $(K_p^{n})$ of non-negative real numbers similar to that in
\cite[Section 4]{FirLeu25}:
\begin{equation} \label{def-kzp-knp}
K_p^{0} = \max \lbr \norm{x_0 - p}, \frac{\norm{f(p) - p}}{1 - \rho}, \norm{p} \rbr, \quad
K_p^{n+1} = K_p^{n} + (1 - \alpha_n - \beta_n - \delta_n)\norm{p} + \norm{r_n}. 
\end{equation}

\blem\label{upper-bds}
For all $n \in \N$, 
\[\norm{\xn - p}, \norm{f(\xn) - p} \leq K_p^n.\]
\elem
\begin{proof}
The proof is by induction on $n$.

$n = 0$: We have that, by \eqref{def-kzp-knp}, $\norm{x_0 - p} \leq \Kpz$, and, since $f$ is a 
$\rho-$contraction,
\[\norm{f(x_0) - p}\leq \norm{f(x_0) - f(p)} + \norm{f(p) - p} \leq \rho \norm{x_0 - p} + \norm{f(p) - p} \stackrel{\eqref{def-kzp-knp}}{\leq} \Kpz.\]
$n \Rightarrow n + 1$: Use \eqref{xn-p-ineq-viscosity}, the induction hypothesis, and \eqref{def-kzp-knp} to get that
\begin{align*}
\norm{\xnp - p} & \leq \delta_n \norm{f(\xn) - p} + (1 - \delta_n) \norm{\xn - p} + (1 - \alpha_n - \beta_n - \delta_n)\norm{p} + \norm{r_n} \\
& \leq  K_p^n + (1 - \alpha_n - \beta_n - \delta_n)\norm{p} + \norm{r_n} = K_p^{n+1}.
\end{align*}
Then $\norm{f(\xnp) - p} \leq K_p^{n+1}$ follows easily.
\end{proof}

\subsection{Main technical lemma}

\blem
For all $n\in \N$, 
\begin{align}
\norm{x_{n+2} - x_{n+1}} & \le (1-(1-\rho)\delta_{n+1})\norm{x_{n+1} - x_n} +
 |\delta_{n+1}-\delta_n|\norm{f(x_n)} + |\alpha_{n+1}-\alpha_n|\norm{x_n} \label{ineq-rec-xnp-xn}\\
& \quad  +|\beta_{n+1}-\beta_n|\norm{Tx_n} + \norm{r_{n+1}-r_n} \nonumber, \\
\norm{x_n - Tx_n } & \leq \norm{x_{n+1}-x_n} + (1-\beta_n)\norm{Tx_n} + \alpha_n \norm{x_n} 
+ \delta_n \norm{f(x_n)} + \norm{r_n} \label{ineq-Txn-xn-xnp-xn}.
\end{align}
\elem
\begin{proof}
First note that
\begin{align*}
x_{n+2} - x_{n+1} & = (\delta_{n+1} f(x_{n+1}) + \alpha_{n+1} x_{n+1} + \beta_{n+1} T x_{n+1} + 
r_{n+1}) \\ 
& \quad - (\delta_n f(x_n) + \alpha_n x_n + \beta_n Tx_n+ r_n) \\
& = \delta_{n+1}(f(x_{n+1})-f(x_n))+  (\delta_{n+1}-\delta_n)f(x_n)+  
\alpha_{n+1}(x_{n+1}-x_n)   + (\alpha_{n+1}-\alpha_n)x_n \\
 & \quad  +  \beta_{n+1}(Tx_{n+1} - Tx_n) + (\beta_{n+1}-\beta_n)Tx_n + (r_{n+1}-r_n).
\end{align*}
It follows that 
\begin{align*}
\norm{x_{n+2} - x_{n+1}} & \le \delta_{n+1}\norm{f(x_{n+1})-f(x_n)} +
|\delta_{n+1}-\delta_n|\norm{f(x_n)} + |\alpha_{n+1}-\alpha_n|\norm{x_n}  \\
& \quad + \alpha_{n+1}\norm{x_{n+1}-x_n} + \beta_{n+1}\norm{Tx_{n+1} - Tx_n} + |\beta_{n+1}-\beta_n|\norm{Tx_n} \\
& \quad + \norm{r_{n+1}-r_n} \\
& \leq  (\delta_{n+1}\rho+\alpha_{n+1}+\beta_{n+1})\norm{x_{n+1} - x_n} 
+ |\delta_{n+1}-\delta_n|\norm{f(x_n)} + |\alpha_{n+1}-\alpha_n|\norm{x_n}  \\
& \quad  +|\beta_{n+1}-\beta_n|\norm{Tx_n} + \norm{r_{n+1}-r_n}\\
& \leq (\delta_{n+1}\rho+1-\delta_{n+1})\norm{x_{n+1} - x_n} 
+ |\delta_{n+1}-\delta_n|\norm{f(x_n)} + |\alpha_{n+1}-\alpha_n|\norm{x_n}  \\
& \quad  +|\beta_{n+1}-\beta_n|\norm{Tx_n} + \norm{r_{n+1}-r_n}, 
\end{align*}
since $\alpha_{n+1} + \beta_{n+1} + \delta_{n+1} \leq 1$, so 
$\alpha_{n+1} + \beta_{n+1} \leq 1- \delta_{n+1}$.
Thus, \eqref{ineq-rec-xnp-xn} holds. \\
Furthermore, we have that 
\begin{align*}
Tx_n - x_n & = (x_{n+1}-x_n) + Tx_n-x_{n+1} = (x_{n+1}-x_n) + Tx_n-(\delta_n f(x_n) + \alpha_nx_n + \beta_n Tx_n+ r_n) \\
& = (x_{n+1}-x_n) + (1-\beta_n)Tx_n - \alpha_n x_n - \delta_n f(x_n)-r_n,
\end{align*}
hence \eqref{ineq-Txn-xn-xnp-xn} follows easily. 
\end{proof}

\subsection{Quantitative hypotheses} \label{quantitative-hypotheses}

The following quantitative hypotheses are considered for the sequences $(\alpha_n)$, $(\beta_n)$, $(\delta_n)$, $(r_n)$:\\

\begin{tabular}{ll}
(Q0) & \quad $\Mabd \in \N$ is an upper bound on $\sum\limits_{n=0}^\infty \lp 1 - \alpha_n - \beta_n - \delta_n \rp$;\\[1mm]
(Q1$\delta_n$) & \quad $\sum\limits_{n=0}^\infty \delta_n = \infty$ with rate of divergence $\rateoned$; \\[1mm]
(Q1$^*\delta_n$) &\quad $\rateproddelta:\N \times \N \to \N$ satisfies the following: for all $m,k \in \N$, \\[1mm]
& \quad $\rateproddelta(m,k) \geq m$ and $\prod\limits_{i=m}^{\rateproddelta(m,k)}(1-(1-\rho)\delta_{i+1})\leq \frac1{k+1}$ ;\\[1mm]
(Q2$\delta_n$) & \quad $\sum\limits_{n=0}^{\infty} \abs{\delta_n-\delta_{n+1}} <\infty$ with Cauchy modulus $\ratetwod$;\\[1mm]
(Q3$\delta_n$) & \quad  $\lim\limits_{n\to\infty} \delta_n=0$ with rate of convergence $\ratethreed$;\\[1mm]
(Q1$\alpha_n$) & \quad $\sum\limits_{n=0}^{\infty} \abs{\alpha_n-\alpha_{n+1}} <\infty$ with Cauchy modulus $\rateonea$;\\[1mm]
(Q1$\beta_n$) & \quad $\sum\limits_{n=0}^{\infty} \abs{\beta_n-\beta_{n+1}} <\infty$ with Cauchy modulus $\rateoneb$;\\[1mm]
(Q2$\beta_n$) & \quad $\lim\limits_{n\to\infty} \beta_n=1$ with rate of convergence $\ratetwob$;\\[1mm]
(Q1$r_n$) & \quad $\sum\limits_{n=0}^{\infty} \norm{r_n - r_{n+1}} <\infty$ with Cauchy modulus $\rateoner$;\\[1mm]
(Q2$r_n$) & \quad $\sum\limits_{n=0}^{\infty} \norm{r_n} <\infty$ with Cauchy modulus $\ratetwor$.
\end{tabular}

\mbox{}

\begin{lemma}\label{H1deltan-prop}
Assume that (Q1$\delta_n$) holds. Then 
\be
\item\label{H1deltan-prop-1} $\rateoned(n) \geq n-1$ for all $n \in \N$. 
\item\label{H1deltan-prop-2} $\sum\limits_{n=0}^\infty (1-\rho)\delta_{n+1}$ diverges with rate 
$\theta(n)=\rateoned^+\lp\ceil{\frac{n}{1-\rho}}+1\rp\remin 1$.
\ee
\end{lemma}
\begin{proof}
\eqref{H1deltan-prop-1} is \cite[Lemma 2.3(i)]{FirLeu26}, as $(\delta_n) \subseteq [0,1]$. Furthermore, \eqref{H1deltan-prop-2} is
\cite[Lemma 3.5.(iii)]{FirLeu26}.
\end{proof}

For the rest of the section, we assume that (Q0) and (Q2$r_n$) hold. Let  $\Mr \in \N$ satisfy the following: 
\begin{center}
$\Mr=0$ if $ r_n = 0$ for all $n \in \N$, and $\Mr$ is an upper bound  on $\sum\limits_{n = 0}^{\infty}\norm{r_n}$ otherwise.
\end{center}
By Lemma \ref{cauchy-mod-up-bd-conv-rate}.\eqref{cauchy-mod-up-bd-conv-rate-1}, such an $\Mr$ can be computed using $\ratetwor$.

For all $p \in \operatorname{Fix}(T)$, denote
\begin{equation}\label{def-Kp}
\Kp = \ceil{(2 + \Mabd) \Kpz} + \Mr+1. 
\end{equation}

\blem \label{explicit-upper-b-xn-txn}
For all  $p \in \operatorname{Fix}(T)$, $n \in \N$, the following hold:
\be
\item\label{Kpn-Mr-MAbd} $K_p^n \leq (1 + \Mabd)\Kpz + \Mr \leq \Kp$;
\item\label{upper-bound-xnpup}
$\norm{\xn - p}, \norm{f(\xn) - p}, \norm{T\xn - p} \leq (1 + \Mabd)\Kpz + \Mr$;
\item\label{bd-xn-txn} $\norm{\xn}, \norm{T\xn}, \norm{f(\xn)} \leq \Kp$;
\item\label{bd-xnpxn} $\norm{\xn - f(\xn)}, \norm{\xnp - \xn}\leq 2\Kp$.
\ee
\elem
\begin{proof}
\be
\item The case $n=0$ is obvious. For $n\in \Ns$ we get that 
\begin{align*}
K_p^n & \stackrel{\eqref{def-kzp-knp}}{=} \Kpz + \sum\limits_{i = 0}^{n - 1} (1 - \alpha_i - \beta_i - \delta_i) \norm{p} + 
\sum\limits_{i = 0}^{n - 1} \norm{r_i} \leq \Kpz + \norm{p} \Mabd + \Mr \\
& \stackrel{\eqref{def-kzp-knp}}{\leq}  (1 + \Mabd)\Kpz + \Mr.
\end{align*}
\item Apply Lemma \ref{upper-bds} and \eqref{Kpn-Mr-MAbd} to get that 
$\norm{\xn - p}, \norm{f(\xn) - p} \leq (1 + \Mabd)\Kpz + \Mr$. Furthermore, 
$\norm{T\xn - p} \leq \norm{\xn - p} \leq (1 + \Mabd)\Kpz + \Mr$.
\item Apply \eqref{upper-bound-xnpup} and the fact that, by  \eqref{def-kzp-knp}, $\norm{p} \leq \Kpz$.
\item By  \eqref{bd-xn-txn}.
\ee
\end{proof}

\blem
For all $p \in \operatorname{Fix}(T)$, $n \in \N$, 
\begin{align}
\norm{x_{n+2} - x_{n+1}} & \leq (1-(1-\rho)\delta_{n+1})\norm{x_{n+1} - x_n} + 
\Kp d_n  + \norm{r_{n+1}-r_n}, \label{xnpuxn-ineq-main}\\
\norm{Tx_n - x_n } & \leq \norm{x_{n+1}-x_n} + 2  \Kp\lp 1-\beta_n\rp + \norm{r_n}. \label{xnTxn-ineq-main}
\end{align}
where $d_n=|\alpha_{n+1}-\alpha_n| + |\beta_{n+1}-\beta_n| + |\delta_{n+1}-\delta_n|$. 
\elem
\begin{proof}
Apply \eqref{ineq-rec-xnp-xn},  \eqref{ineq-Txn-xn-xnp-xn}, and Lemma \ref{explicit-upper-b-xn-txn}.\eqref{bd-xn-txn}.
\end{proof}

\section{Main theorems}

The main quantitative results of the paper are Theorems \ref{thm-t-as-reg-quant-Q1} and  \ref{thm-t-as-reg-quant-Q1s}, which provide 
uniform rates of asymptotic regularity and $T$-asymptotic regularity for  the inexact generalized Halpern iteration $(x_n)$.
These rates are computed for general parameter sequences $(\alpha_n)$, $(\beta_n)$, $(\delta_n)$, and $(r_n)$ satisfying suitable 
quantitative hypotheses. The main qualitative result is Theorem \ref{main-thm-t-as-reg-qualitative}, which proves that $(x_n)$ is asymptotically 
regular and $T$-asymptotically regular. This result is obtained as an immediate corollary of the quantitative theorems.
Furthermore, as consequences of our main results, we obtain both quantitative and qualitative results for the Kanzow--Shehu iteration 
defined by \eqref{KanShe-iteration} and for the sequential averaging method (SAM). 

In the following, let $X$  be a normed space, $T:X \to X$ a nonexpansive mapping with fixed points, 
$f:X \to X$ a $\rho$-contraction for some $\rho \in [0,1)$, $(r_n)$ a sequence in $X$, 
$(\alpha_n)$, $(\beta_n)$, $(\delta_n)$  sequences in $[0,1]$, and let $(\xn)$ be the inexact 
generalized Halpern iteration defined by \eqref{inexact-gen-Halpern}.

\bthm \label{thm-t-as-reg-quant-Q1}
Assume that
\be
\item  (Q0), (Q1$\delta_n$), (Q2$\delta_n$), (Q1$\alpha_n$), (Q1$\beta_n$), (Q1$r_n$), and (Q2$r_n$) hold;
\item  $K_p$ is defined by \eqref{def-Kp} for some $p\in \operatorname{Fix}(T)$. 
\ee 
Define $\chi: \N\to \N$ by 
\begin{equation}\label{thm-t-as-reg-quant-def-chi}
\chi(k) = \max\lbr\ratetwod(6\Kp(k+1)-1), \rateonea(6\Kp(k+1)-1),\rateoneb(6\Kp(k+1)-1), 
\rateoner(2k+1)\rbr.
\end{equation}
Then
\be
\item\label{thm-t-as-reg-quant-Q1-1} $(\xn)$ is asymptotically regular with rate
\begin{equation}
\Phi(k) = \rateoned^+\lp \ceil{\frac{\chi(2k+1)+1+ \ceil{\ln(4\Kp(k+1))}}{1-\rho}}+1\rp.
\end{equation}

\item\label{thm-t-as-reg-quant-Q1-2} If, furthermore, (Q2$\beta_n$) holds, 
then  $(\xn)$ is $T$-asymptotically regular with rate
\begin{equation}
\Psi(k) = \max \lbr \Phi(3k+2), \ratetwob\lp 6\Kp(k+1) - 1 \rp, \ratetwor(3k+2) + 1 \rbr.
\end{equation}
\ee
\ethm
\begin{proof}
\be
\item Let us denote $L:=2\Kp$ and, for all $n\in\N$, 
\[ s_n := \norm{\xnp - \xn},  \qquad  a_n := (1-\rho)\delta_{n+1}, \qquad 
c_n := \Kp d_n + \norm{r_{n+1}-r_n}, 
\]
where
$d_n:= |\alpha_{n+1}-\alpha_n| + |\beta_{n+1}-\beta_n|+|\delta_{n+1}-\delta_n|$. \\
We have that $s_{n+1}\leq (1-a_n)s_n + c_n$ for all  $n\in\N$ (by \eqref{xnpuxn-ineq-main}) and that 
$L$ is an upper bound on $(s_n)$ (by Lemma \ref{explicit-upper-b-xn-txn}.\eqref{bd-xnpxn}).
Apply (Q1$\delta_n$) and Lemma \ref{H1deltan-prop}.\eqref{H1deltan-prop-2} to get that 
$\sum\limits_{n=0}^\infty a_n$ diverges with rate 
\[\theta(n)=\rateoned^+\lp\ceil{\frac{n}{1-\rho}}+1\rp\remin 1.\]
Furthermore, one can easily see that $\sum\limits_{n=0}^\infty d_n$ is Cauchy with modulus
\[\varphi(k)= \max \lbr \ratetwod(3k+2), \rateonea(3k+2), \rateoneb(3k+2)\rbr.\]
Hence, we get by Lemma \ref{series-Cauchy-modulus-linear-comb}, that $\sum\limits_{n=0}^\infty c_n$ is Cauchy with modulus $\chi$ 
given by \eqref{thm-t-as-reg-quant-def-chi}.

Apply now Proposition \ref{quant-Xu}.\eqref{quant-Xu-1} to obtain that $(\xn)$ is 
asymptotically regular with rate
\begin{align*}
\Sigma(k) & = \lp \rateoned^+ \lp \ceil{\frac{\chi(2k+1)+1+\lceil \ln(4\Kp(k+1))\rceil}{1-\rho}}+1\rp \remin 1 \rp + 1 = \lp  \Phi(k)  \remin 1 \rp + 1 \\
 & = \Phi(k), \quad \text{since,  by Lemma \ref{H1deltan-prop}.\eqref{H1deltan-prop-1}, } \Phi(k) \geq 1.
\end{align*}

\item First, let us remark that by (Q2$r_n$) and 
Lemma \ref{cauchy-mod-up-bd-conv-rate}.\eqref{cauchy-mod-up-bd-conv-rate-2}, we get that 
\begin{center}
($\ast$) \quad $\lim\limits_{n\to \infty} \norm{r_n} = 0$ with rate of convergence $\psi(k)=\ratetwor(k)+1$.
\end{center}
Let $n \geq \Psi(k)$. Then, by \eqref{thm-t-as-reg-quant-Q1-1}, (Q2$\beta_n$), and ($\ast$), we get that 
\[\norm{x_{n+1}-x_n} \leq \frac1{3(k + 1)}, \quad 1-\beta_n \leq \frac1{6\Kp (k+1)}, 
\quad \norm{r_n} \leq \frac1{3(k + 1)}.\]
\ee
It follows that 
\begin{align*}
\norm{Tx_n - x_n } & \stackrel{\eqref{xnTxn-ineq-main}}{\leq} \norm{x_{n+1}-x_n} + 2 \Kp
\lp 1-\beta_n\rp  + \norm{r_n}  \leq  \frac{1}{k+1}.
\end{align*}
\end{proof}

\bthm \label{thm-t-as-reg-quant-Q1s}
In the hypotheses of Theorem \ref{thm-t-as-reg-quant-Q1}, suppose that (Q1$^*\delta_n$) holds  	
instead of (Q1$\delta_n$). Then 
\be
\item\label{thm-t-as-reg-quant-Q1s-1} $(\xn)$ is asymptotically regular with rate
\begin{equation}
\Phi^*(k)=\rateproddelta \lp\chi(2k+1)+1, 4\Kp(k+1)-1 \rp + 1.
\end{equation}
\item\label{thm-t-as-reg-quant-Q1s-2}  If, furthermore, (Q2$\beta_n$) holds, 
then  $(\xn)$ is $T$-asymptotically regular with rate
\begin{equation}
\Psi^*(k)= \max \lbr \Phi^*(3k+2), \ratetwob\lp 6\Kp (k+1) - 1 \rp, \ratetwor(3k+2) + 1 \rbr.
\end{equation}
\ee
\ethm
\begin{proof} 
\eqref{thm-t-as-reg-quant-Q1s-1} Follow the proof of Theorem \ref{thm-t-as-reg-quant-Q1}.\eqref{thm-t-as-reg-quant-Q1-1} and apply 
Proposition \ref{quant-Xu}.\eqref{quant-Xu-2} instead of
Proposition \ref{quant-Xu}.\eqref{quant-Xu-1}. \\
\eqref{thm-t-as-reg-quant-Q1s-2} Replace $\Phi$ with with $\Phi^*$in the proof of 
Theorem \ref{thm-t-as-reg-quant-Q1}.\eqref{thm-t-as-reg-quant-Q1-2}.
\end{proof}

Thus, the rates $\Phi$, $\Phi^*$, $\Psi$, $\Psi^*$ depend on $X$, $T$, $f$ only via $\Kpz$ 
given by \eqref{def-kzp-knp}, and on the parameter sequences $(\alpha_n)$, $(\beta_n)$, $(\delta_n)$, and $(r_n)$ through the 
moduli appearing in the quantitative hypotheses they satisfy. In particular, if all these moduli are polynomial, 
then the rates $\Phi$, $\Phi^*$, $\Psi$, $\Psi^*$ are also polynomial. 

Our main theorems are quantitative ($T$-)asymptotic regularity results for the inexact generalized Halpern iteration $(x_n)$.
As an immediate consequence, we obtain a qualitative ($T$-)asymptotic regularity result for this iteration.

\bthm\label{main-thm-t-as-reg-qualitative}
Assume that the following hold:
\be
\item The series $\sum\limits_{n=0}^\infty \lp 1 - \alpha_n - \beta_n - \delta_n \rp$, \, 
$\sum\limits_{n=0}^{\infty} \abs{\delta_n-\delta_{n+1}}$, \, 
$\sum\limits_{n=0}^{\infty} \abs{\alpha_n-\alpha_{n+1}}$,  \, $\sum\limits_{n=0}^{\infty} 
\abs{\beta_n-\beta_{n+1}}$,  \, \\
$\sum\limits_{n=0}^{\infty} \norm{r_n}$,  \, $\sum\limits_{n=0}^{\infty} \norm{r_n - r_{n+1}}$
 converge;
\item $\lim\limits_{n\to\infty} \beta_n=1$; 
\item $\lp \sum\limits_{n=0}^\infty \delta_n \text{ diverges} \rp$ or 
$\lp \text{for all }m\in \N, \, 
\prod\limits_{i=m}^{\infty}(1-(1-\rho)\delta_{i+1})=0\rp$.
\ee
Then $(x_n)$ is asymptotically regular and $T$-asymptotically regular, that is 
\[\lim\limits_{n\to\infty} \norm{x_n-x_{n+1}} = \lim\limits_{n\to\infty} \norm{x_n-Tx_n}=0. \]
\ethm

We conclude this section by presenting some consequences of the main theorems, computing 
rates for special cases of the inexact generalized Halpern iteration $(x_n)$.

\subsection{Kanzow--Shehu iteration given by \eqref{KanShe-iteration}}

If we let $f(x)=u\in X$ for all $x\in X$, so that $\rho =0$, then $(x_n)$ reduces to the 
Kanzow--Shehu iteration 
\eqref{KanShe-iteration}:
\[\xnp = \delta_n u + \alpha_n \xn + \beta_n T \xn + r_n.\]
Then $\Kp$ is defined by \eqref{def-Kp} with $\Kpz=\max \lbr \norm{x_0 - p}, \norm{u - p}, \norm{p} \rbr$. 
Theorem \ref{thm-t-as-reg-quant-Q1} holds with 
\[\Phi(k) = \rateoned^+\lp \chi(2k+1)+\ceil{\ln(4\Kp(k+1))}+2\rp,\]
and Theorem  \ref{thm-t-as-reg-quant-Q1s} holds with $\rateproddelta$ satisfying for all $m,k\in \N$, 
\[
\rateproddelta(m,k) \geq m \text{ and }\prod\limits_{i=m}^{\rateproddelta(m,k)}(1-\delta_{i+1})\leq \frac{1}{k+1}.\]

Furthermore, in the hypotheses of Theorem \ref{main-thm-t-as-reg-qualitative}, one takes 
\begin{center}
$\lp \text{for all }m\in \N, \, \prod\limits_{i=m}^{\infty}(1-\delta_{i+1})=0\rp$ instead of $\lp \text{for all }m\in \N, \, 
\prod\limits_{i=m}^{\infty}(1-(1-\rho)\delta_{i+1})=0\rp$.
\end{center}

By letting $\rho=0$ in Theorem \ref{main-thm-t-as-reg-qualitative}, we get a ($T$-)asymptotic regularity 
result for this iteration.

\subsection{The generalized Halpern iteration}

Suppose that $r_n = 0$ for all $n \in \N$, hence 
\[\xnp = \delta_n f(x_n) + \alpha_n \xn + \beta_n T\xn \quad \text{for all } n \in \N.\]
Then  $\Mr = 0$, so $\Kp = \ceil{(2 + \Mabd) \Kpz}+1$. 
As (Q1$r_n$), (Q2$r_n$) are satisfied with $\rateoner(k)= \ratetwor(k)=0$ for all $k \in \N$, we get that Theorems \ref{thm-t-as-reg-quant-Q1}, 
\ref{thm-t-as-reg-quant-Q1s} hold with 
\begin{align*}
\chi(k) &= \max\lbr\ratetwod(6\Kp(k+1)-1), \rateonea(6\Kp(k+1)-1),\rateoneb(6\Kp(k+1)-1)\rbr, \\
\Psi(k) &= \max \lbr \Phi(3k+2), \ratetwob\lp 6\Kp(k+1) - 1 \rp\rbr, \\
\Psi^*(k) & = \max \lbr \Phi^*(3k+2), \ratetwob\lp 6\Kp(k+1) - 1 \rp\rbr.
\end{align*}
Obviously, in Theorem  \ref{main-thm-t-as-reg-qualitative}, the hypotheses that the series  $\sum\limits_{n=0}^{\infty} \norm{r_n - r_{n+1}}$, 
$\sum\limits_{n=0}^{\infty} \norm{r_n}$ converge can be removed, as they are superfluous.

If we assume, moreover, that $\alpha_n + \beta_n + \delta_n = 1$ for all $n \in \N$, then (Q0) holds with 
$\Mabd=0$, so $\Kp = \ceil{2\Kpz}+1$, and we can also omit the 
hypothesis that $\sum\limits_{n=0}^\infty \lp 1 - \alpha_n - \beta_n - \delta_n \rp$ converges in 
Theorem  \ref{main-thm-t-as-reg-qualitative}.

\subsection{Sequential averaging method (SAM)}\label{subsec-VAM}

Assume that for all $n \in \N$, $\alpha_n =0$, $\beta_n =1- \delta_n$, and $r_n=0$. Then
\[\xnp = \delta_n f(x_n) +  (1-\delta_n) T\xn \quad \text{for all } n \in \N.\]
Thus,  $(x_n)$ is the sequential averaging method (SAM). We get the following quantitative 
asymptotic regularity  result for SAM. 

\bprop\label{SAM-quant-t-as-reg}
Assume that (Q2$\delta_n$) and (Q3$\delta_n$) are satisfied, and that either (Q1$\delta_n$) or (Q1$^*\delta_n$) holds. 
Let, for some $p\in \operatorname{Fix}(T)$, 
\begin{center}
$\Kpz = \max \lbr \norm{x_0 - p}, \frac{\norm{f(p) - p}}{1 - \rho}, \norm{p} \rbr$, \quad
$\Kp = \ceil{2\Kpz}+1$.
\end{center}
Define $\chi: \N \to \N, \,\,  \chi(k) = \ratetwod(6\Kp(k+1)-1)$. 

Then $(\xn)$ is asymptotically regular with rate
\begin{align*}
\Phi_0(k) = \begin{cases} 
\rateoned^+\lp \ceil{\frac{\chi(2k+1)+1+ \ceil{\ln(4\Kp(k+1))}}{1-\rho}}+1\rp & \text{if (Q1$\delta_n$) holds},\\[2mm]
\rateproddelta \lp\chi(2k+1)+1, 4\Kp(k+1)-1 \rp + 1 & \text{if (Q1$^*\delta_n$) holds},
\end{cases}
\end{align*}
and $T$-asymptotically regular with rate
\[
\Psi_0(k)= \max \lbr \Phi_0(3k+2), \ratethreed\lp 6\Kp (k+1) - 1 \rp\rbr.
\]
\eprop
\begin{proof}
Remark that $\Mabd=\Mr=0$, so $\Kp = \ceil{2\Kpz}+1$. Furthermore, (Q1$\alpha_n$), (Q1$r_n$), (Q2$r_n$) hold with  
$\rateonea(k)=\rateoner(k)= \ratetwor(k)=0$ for all $k \in \N$,
(Q1$\beta_n$) is (Q2$\delta_n$) (so $\rateoneb=\ratetwod$), and (Q2$\beta_n$)  is (Q3$\delta_n$) (so $\ratetwob=\ratethreed$). 
Apply now Theorems \ref{thm-t-as-reg-quant-Q1}, \ref{thm-t-as-reg-quant-Q1s}. 
\end{proof}

The corresponding qualitative asymptotic regularity result is obtained as an immediate consequence.

\bprop\label{SAM-t-as-reg}
Assume that $\sum\limits_{n=0}^{\infty} \abs{\delta_n-\delta_{n+1}}< \infty$ and $\lim\limits_{n\to\infty} \delta_n=0$. Suppose, furthermore, that 
\begin{center}
either  $\sum\limits_{n=0}^\infty \delta_n =\infty $ or $\lp \text{for all }m\in \N, \, 
\prod\limits_{i=m}^{\infty}(1-(1-\rho)\delta_{i+1})=0\rp$ holds.
\end{center}
Then $(x_n)$ is asymptotically regular and $T$-asymptotically regular. 
\eprop

If we assume further that $\rho =0$ and that $f(x)=u\in X$ for all $x\in X$, $(x_n)$ reduces to the well-known Halpern iteration. 
Then Proposition \ref{SAM-quant-t-as-reg} holds for the Halpern iteration if we make the following changes: 
take 
\[\Kpz=\max \lbr \norm{x_0 - p}, \norm{u - p}, \norm{p} \rbr\]
 and 
 \[\Phi_0(k)=\rateoned^+\lp \chi(2k+1)+\ceil{\ln(4\Kp(k+1))}+2\rp\] 
in the case that (Q1$\delta_n$) holds.

\section{Examples}\label{section-examples}

In this section, we present three examples of concrete parameter sequences $(\alpha_n)$, $(\beta_n)$, 
$(\delta_n)$, and $(r_n)$. We show that the first two examples satisfy the hypotheses of our 
main quantitative theorems, and we apply these theorems to compute rates of ($T$-)asymptotic regularity 
for the (inexact) generalized Halpern iteration $(x_n)$. In the first example, 
$\alpha_n + \beta_n + \delta_n = 1$ for all $n\in \N$, and in the second one $\alpha_n + \beta_n + \delta_n < 1$ 
for all $n\in \N$. 
Applying Theorem \ref{thm-t-as-reg-quant-Q1} yields exponential rates, while Theorem \ref{thm-t-as-reg-quant-Q1s}
provides improved polynomial (quadratic) rates. This confirms an observation first made by 
Kohlenbach \cite{Koh11} for the Halpern iteration in normed spaces. 
The third example allows us to apply a well-known lemma due to Sabach and Shtern \cite{SabSht17}, 
and to obtain, as a consequence, linear rates.

\subsection{Example 1}\label{example1}

Assume that for all $n \in \N$,
\[
\alpha_n = \frac{\lambda}{n + J}, \quad \beta_n = 1 - \lp \lambda +  1 \rp \frac{1}{n + J}, \quad 
\delta_n = \frac1{n + J}, \quad r_n = \frac{1}{(n + P)^2} r^*,
\]
where $J,P\in \Ns$, $\lambda \in \lb 0, J - 1 \rb$, and $r^* \in X$. Obviously, for all 
$n \in \N$, $\alpha_n, \beta_n \in \lb 0, 1 \rp$, $\delta_n \in \lp 0, 1 \rb$, and 
$\alpha_n + \beta_n + \delta_n = 1$. 

\blem\label{example1-quant-hyp-1}
\be
\item\label{example1-quant-hyp-1-Q0}  (Q0) holds with $\Mabd=0$.
\item\label{example1-quant-hyp-1-Q2d} (Q2$\delta_n$) holds with $\ratetwod(k)=k+1$.
\item\label{example1-quant-hyp-1-Q1a} (Q1$\alpha_n$) holds with $\rateonea(k) = \ceil{\lambda(k+1)}$.
\item\label{example1-quant-hyp-1-Q1b} (Q1$\beta_n$) holds with $\rateoneb(k)= \ceil{\lp\lambda +  1\rp(k+1)}$.
\item\label{example1-quant-hyp-1-Q2b} (Q2$\beta_n$) holds with $\ratetwob(k)=\ceil{\lp\lambda +  1\rp(k+1)}$.  
\item\label{example1-quant-hyp-1-Q1r} (Q1$r_n$) holds with $\rateoner(k)= \ceil{\sqrt{\norm{r^*}(k+1)}}$.
\item\label{example1-quant-hyp-1-Q2r} (Q2$r_n$) holds with $\ratetwor(k) = \ceil{\norm{r^*}(k+1)}$.
\ee
\elem
\begin{proof}
\eqref{example1-quant-hyp-1-Q0} is obvious, as $1 - \alpha_n - \beta_n - \delta_n =0$ for all $n\in \N$. 
As $\lp \alpha_n\rp$ and  $\lp \delta_n\rp$ are decreasing, and 
$\lp \beta_n\rp$ is increasing, we have that for all $n\in \N$, $l\in \N^*$,
\begin{align*}
\sum\limits_{i= n+1}^{n+l} \abs{\delta_i - \delta_{i+1}} & = 
\sum\limits_{i= n+1}^{n+l}\lp \delta_i - \delta_{i+1}\rp  = \delta_{n+1} - \delta_{n+l+1} < \delta_{n+1} = \frac{1}{n + J+1}, \\
\sum\limits_{i= n+1}^{n+l} \abs{\alpha_i - \alpha_{i+1}}  & = 
\sum\limits_{i= n+1}^{n+l}\lp \alpha_i - \alpha_{i+1}\rp = \alpha_{n+1} - \alpha_{n+l+1} < \alpha_{n+1} = \frac{\lambda}{n+J+1}, \\
\sum\limits_{i= n+1}^{n+l} \abs{\beta_i - \beta_{i+1}} & = 
\sum\limits_{i= n+1}^{n+l}\lp \beta_{i+1} -\beta_i \rp = \beta_{n+l+1} - \beta_{n+1} 
 < \lp \lambda +  1 \rp \frac1{n + J+1}, \\
1 - \beta_n & =  \lp \lambda +  1 \rp \frac1{n + J}.
\end{align*}
Apply now Lemma \ref{useful-Cauchy-rate}.\eqref{useful-Cauchy-rate-2} to get that \eqref{example1-quant-hyp-1-Q2d}, 
\eqref{example1-quant-hyp-1-Q1a},  \eqref{example1-quant-hyp-1-Q1b}, and  \eqref{example1-quant-hyp-1-Q2b} hold. 

Furthermore, \eqref{example1-quant-hyp-1-Q1r} is obtained by an application  of Lemma \ref{useful-Cauchy-rate}.\eqref{useful-Cauchy-rate-3}, as 
for all $n\in \N$, $l\in \N^*$, 
\begin{align*}
\sum\limits_{i= n+1}^{n+l}\norm{r_i - r_{i+1}}  & = 
\sum\limits_{i= n+1}^{n+l} \lp \frac{\norm{r^*}}{(i + P)^2}  -\frac{\norm{r^*}}{(i + P+1)^2} \rp 
 = \frac{\norm{r^*}}{(n+ P+1)^2}  - \frac{\norm{r^*}}{(n+ P+l)^2} \\
 & \leq \frac{\norm{r^*}}{(n+ P+1)^2}.
\end{align*}
Finally, \eqref{example1-quant-hyp-1-Q2r} holds,  by Lemma \ref{useful-Cauchy-rate}.\eqref{useful-Cauchy-rate-1}.
\end{proof}

\blem\label{example1-quant-hyp-2}
\be
\item\label{example1-quant-hyp-2-1} (Q1$\delta_n$) holds with $\rateoned(n) = \ceil{J \exp(n)}-J$.
\item\label{example1-quant-hyp-2-2} (Q1$^*\delta_n$) holds with 
$\rateproddelta(m,k) = \ceil{(m + J+1) (k+1)^{\frac{1}{1 - \rho}}}-J-1$.
\ee
\elem
\begin{proof}
\be
\item Let $n \in \N$.  Remark first that $\rateoned(n) \geq 0$. Moreover, 
\begin{align*}
\sum\limits_{i = 0}^{\rateoned(n)} \frac{1}{i + J} & = \sum\limits_{m = J}^{\rateoned(n) + J} \frac{1}{m} 
\geq \ln(\rateoned(n) + J + 1) - \ln J \geq \ln \frac{\ceil{J \exp(n)}}{J} 
\geq  \ln \lp\exp(n)\rp = n. 
\end{align*}
\item Let $k, m \in \N$. As $\rho \in [0,1)$, we have that $1 - \rho \in (0,1]$, and hence 
$(k + 1)^{\frac{1}{1 - \rho}} \geq k+1\geq 1$, so $ \rateproddelta(m,k) \geq m$. 
Denote
\[
A := \prod\limits_{i = m}^{\rateproddelta(m,k)} (1 - (1 - \rho) \delta_{i+1}) = 
\prod\limits_{i = m}^{\rateproddelta(m,k)} \lp 1 - \frac{1 - \rho}{i + J+1} \rp, 
\text{ and } A_1 := \sum\limits_{i = m}^{\rateproddelta(m,k)} \frac{1}{i + J+1}.
\]
As $\ln x \leq x - 1$ for $x>0$, we get  that 
\begin{align*}
\ln A & = \sum\limits_{i = m}^{\rateproddelta(m,k)} \ln \lp 1 - \frac{1 - \rho}{i + J+1} \rp 
\leq \sum\limits_{i = m}^{\rateproddelta(m,k)} - \frac{1 - \rho}{i + J+1} = - (1-\rho) A_1.
\end{align*}
Since
\begin{align*}
A_1 & = 
\sum\limits_{i = m + J+1}^{\rateproddelta(m,k) + J+1} \frac{1}{i} 
\geq \ln \lp \rateproddelta(m,k) + J + 2\rp - \ln \lp m + J+1\rp   = \ln \frac{\rateproddelta(m,k) + J + 2}{m + J+1} \\
&  \geq\ln (k+1)^{\frac{1}{1 - \rho}}  = \frac{1}{1 - \rho}\ln (k+1), 
\end{align*}
it follows that $\ln A \leq \ln \frac1{k+1}$. Thus, $A \leq  \frac1{k+1}$. 
\ee
\end{proof}

By Lemma \ref{example1-quant-hyp-1}.\eqref{example1-quant-hyp-1-Q0}, $\Mabd=0$, and, by the definition of $\Mr$ and
Lemma \ref{useful-Cauchy-rate}.\eqref{useful-Cauchy-rate-1u}, we can take $\Mr = 2\ceil{\norm{r^*}}$. Then 
\begin{equation} \label{def-ex1-Kp}
\Kp \stackrel{\eqref{def-Kp}}{=}\ceil{(2 + \Mabd) \Kpz} + \Mr+1=\ceil{2\Kpz} + 2\ceil{\norm{r^*}}+1. 
\end{equation}

Applying Theorem \ref{thm-t-as-reg-quant-Q1} and using Lemma \ref{example1-quant-hyp-1}, and
Lemma \ref{example1-quant-hyp-2}.\eqref{example1-quant-hyp-2-1}, we get exponential rates of 
($T$-)asymptotic regularity of $(x_n)$.

\bprop \label{prop1-ex1-1}
Let 
\begin{equation} \label{def-ex1-chi}
\chi(k) = \max\lbr \ceil{6\lambda\Kp(k+1)} +  6\Kp(k+1), \ceil{\sqrt{2\norm{r^*}(k+1)}}\rbr,
\end{equation}
and define
\begin{align*}
\Phi(k) & = \ceil{J \exp \lp \ceil{\frac{\chi(2k+1)+1+ \ceil{\ln(4\Kp(k+1))}}{1-\rho}}+ 1 \rp}- J, \\
\Psi(k) & =  \max \lbr \Phi(3k+2),  \ceil{3\norm{r^*}(k+1)}+1 \rbr.
\end{align*}
Then $\Phi$ is a rate of asymptotic regularity and  $\Psi$ is a rate of $T$-asymptotic regularity of $(x_n)$. 
\eprop

However, if we apply instead Theorem \ref{thm-t-as-reg-quant-Q1s} and use 
Lemma \ref{example1-quant-hyp-2}.\eqref{example1-quant-hyp-2-2}, we get better rates than 
the ones from  Proposition \ref{prop1-ex1-1}. 

\bprop \label{prop1-ex1-2}
Define
\begin{align*}
\Phi^*(k) & = \ceil{(\chi(2k+1)+ J+2) \lp (4\Kp(k+1))^{\frac{1}{1 - \rho}}\rp} - J, \\
\Psi^*(k) & = \max\lbr \Phi^*(3k+2),  \ceil{3\norm{r^*}(k+1)} +1 \rbr.
\end{align*}
where $\chi$ is given by \eqref{def-ex1-chi}.
Then $\Phi^*$ is a rate of asymptotic regularity and  $\Psi^*$ is a rate of $T$-asymptotic regularity of $(x_n)$. 
\eprop

Assume that $\rho = 0$ and let $f(x)=u\in X$ for all $x\in X$. Then, by \eqref{def-ex1-Kp}, 
\[
\Kp = \ceil{2\Kpz} + 2\ceil{\norm{r^*}}+1, \quad \text{with }  
\Kpz=\max \lbr \norm{x_0 - p}, \norm{u - p}, \norm{p} \rbr,
\]
and, by Proposition \ref{prop1-ex1-2}, we get quadratic rates of ($T$-)asymptotic regularity 
of the Kanzow--Shehu iteration \eqref{KanShe-iteration}.

The rates can be further simplified if we additionally assume that $r^* = 0$; in this case $(x_n)$ 
becomes the version without error terms of \eqref{KanShe-iteration}, a particular case of the generalized Halpern iteration.

\begin{corollary}\label{cor-ex1-rrhozero}
Suppose that $r^* = 0$, $\rho = 0$, and $f(x)=u\in X$ for all $x\in X$.  Define
\begin{align*}
\tilde{\Phi}(k) & = 4\Kp \ceil{12 \lambda \Kp (k+1)}(k+1) + 48\Kp^2(k+1)^2 + 4\Kp(J+2)(k+1) - J, \\
\tilde{\Psi}(k) & = \tilde{\Phi}(3k+2),
\end{align*}
where $\Kp = \ceil{2\Kpz}+1$, with $\Kpz=\max \lbr \norm{x_0 - p}, \norm{u - p}, \norm{p} \rbr$.

Then $\tilde{\Phi}$ is a rate of asymptotic regularity and  $\tilde{\Psi}$ is a 
rate of $T$-asymptotic regularity of $(x_n)$.
\end{corollary}
\begin{proof}
Apply \eqref{def-ex1-Kp} and Proposition \ref{prop1-ex1-2}.
\end{proof}

\subsection{Example 2}\label{example2}

Let $J\in \N\setminus \lbr 0,1,2 \rbr$, $P \in \Ns$, $r^* \in X$, $\lp\delta_n\rp$ and $\lp r_n\rp$ be defined as in 
Example 1, and, for all $n \in \N$,
\[
\alpha_n = \frac1{n + J}, \qquad \beta_n = 1 - \frac2{n + J} - \frac1{(n + J)^2}.
\]

We have that for all $n\in \N$, $\alpha_n, \beta_n \in \lp 0, 1 \rp$ and  
\begin{equation}\label{unu-sumabd-2}
1- (\alpha_n + \beta_n + \delta_n) = \frac{1}{(n + J)^2} >0.
\end{equation} 

Apply  \eqref{unu-sumabd-2} and Lemma \ref{useful-Cauchy-rate}.\eqref{useful-Cauchy-rate-1u} to  obtain that (Q0) holds with $\Mabd=2$.
By Lemma \ref{example1-quant-hyp-1}, we have that (Q2$\delta_n$) holds with $\ratetwod(k)=k+1$, 
(Q1$r_n$) holds with $\rateoner(k)= \ceil{\sqrt{\norm{r^*}(k+1)}}$, and (Q2$r_n$) holds with $\ratetwor(k) = \ceil{\norm{r^*}(k+1)}$.
Since $\alpha_n=\delta_n$ for all $n\in \N$, (Q1$\alpha_n$) holds with $\rateonea(k) =  k+1$.

As $\lp \beta_n\rp$ is increasing, we have that for all $n\in \N$, $l\in \N^*$,
\begin{align*}
\sum\limits_{i= n+1}^{n+l} \abs{\beta_i - \beta_{i+1}} & = 
\sum\limits_{i= n+1}^{n+l}\lp \beta_{i+1} -\beta_i \rp = \beta_{n+l+1} - \beta_{n+1} <
\frac2{n+J+1} + \frac1{(n + J+1)^2}  < \frac3{n + J+1}, \\
1 - \beta_n & =  \frac2{n + J} + \frac1{(n + J)^2} < \frac3{n + J}.
\end{align*}
Thus, we can apply Lemma \ref{useful-Cauchy-rate}.\eqref{useful-Cauchy-rate-2} to obtain that 
(Q1$\beta_n$) and (Q2$\beta_n$) hold with $\rateoneb(k)=\ratetwob(k)= 3(k+1)$. \\

Applying Theorem \ref{thm-t-as-reg-quant-Q1} and using 
Lemma \ref{example1-quant-hyp-2}.\eqref{example1-quant-hyp-2-1} we get exponential rates of
($T$-)asymptotic regularity of $(x_n)$.

\bprop \label{prop1-ex2-1}
Define $\Phi$, $\Psi$ as in Proposition \ref{prop1-ex1-1}, with 
\begin{align*} 
\Kp   = \ceil{4\Kpz} + 2\ceil{\norm{r^*}} + 1, \qquad 
\chi(k)  = \max\lbr 12\Kp(k+1),\ceil{\sqrt{2\norm{r^*}(k+1)}}\rbr.
\end{align*}
Then $\Phi$ is a rate of asymptotic regularity and  $\Psi$ is a rate of $T$-asymptotic regularity of $(x_n)$. 
\eprop

As in Example 1, we get better rates of ($T$-)asymptotic regularity of $(x_n)$ if we apply instead 
Theorem \ref{thm-t-as-reg-quant-Q1s} and Lemma \ref{example1-quant-hyp-2}.\eqref{example1-quant-hyp-2-2}. 

\bprop \label{prop1-ex2-2}
Define $\Phi^*$, $\Psi^*$ as in Proposition \ref{prop1-ex1-2}, with $\Kp$, $\chi$ as in Proposition \ref{prop1-ex2-1}. 
Then $\Phi^*$ is a rate of asymptotic regularity and  $\Psi^*$ is a rate of $T$-asymptotic regularity of $(x_n)$. 
\eprop

Applying  Proposition \ref{prop1-ex2-2}, we get a result similar to Corollary \ref{cor-ex1-rrhozero}. 

\begin{corollary}\label{cor-ex2-rrhozero}
Suppose that $r^* = 0$, $\rho = 0$, and $f(x)=u\in X$ for all $x\in X$.  Define
\begin{align*}
\tilde{\Phi}(k)  = 144\Kp^2(k+1)^2 +  4(J+2)\Kp(k+1) -J, \qquad 
\tilde{\Psi}(k)  = \tilde{\Phi}(3k+2),
\end{align*}
where $\Kp = \ceil{4\Kpz}+1$, with $\Kpz=\max \lbr \norm{x_0 - p}, \norm{u - p}, \norm{p} \rbr$.

Then $\tilde{\Phi}$ is a rate of asymptotic regularity and  $\tilde{\Psi}$ is a 
rate of $T$-asymptotic regularity of $(x_n)$.
\end{corollary}

\subsection{Example 3} \label{example3}

Sabach and Shtern proved a lemma on sequences of real numbers \cite[Lemma 3]{SabSht17} and,  
using this lemma, they computed linear rates of ($T$-)asymptotic regularity for SAM. 
The lemma has subsequently been used to derive such linear rates for several Halpern-type iterations 
\cite{CheKohLeu23,LeuPin24,CheLeu25,FirLeu25}. In the sequel, we apply this lemma to derive linear rates of ($T$-)asymptotic regularity 
for the generalized Halpern iteration. To this end, we use the following slightly modified version of Sabach and Shtern's lemma.

\begin{lemma}\label{lemma-SabSht-version}\cite[Lemma 2.8]{LeuPin24} \, \\
Let $L>0$,  $J\geq N\geq 2$,  and $\gamma\in(0,1]$. Assume that $a_n=\frac{N}{\gamma(n+J)}$  and $c_n\leq L$ for all $n\in\N$.
Consider a sequence of nonnegative real numbers $(s_n)$ satisfying the following:
$s_0 \leq L$ and,  for all $n\in\N$,
\[
s_{n+1} \leq (1 - \gamma a_{n+1})s_n + (a_n-a_{n+1})c_n.
\]
Then 
\[
s_n\leq \frac{JL}{\gamma(n+J)} \quad \text{for all~} n\in \N.
\]
\end{lemma}

Let $J, P \in \N$ be such that $P \geq J>\frac{3-\rho}{1-\rho}$, $r^* \in X$, 
$(r_n)$ be defined as in Example 1, and for all $n \in \N$,
\[
\alpha_n = \frac{1}{n + J}, \quad \beta_n = 1 - \frac{3 - \rho}{(1 - \rho) (n+J)}, \quad 
\delta_n = \frac{2}{(1-\rho)(n + J)}.
\]
Obviously, for all $n\in \N$, $\alpha_n,\beta_n, \delta_n \in \lp 0, 1 \rp$, and 
$\alpha_n + \beta_n + \delta_n=1$. Moreover, $(\alpha_n)$ and $(\delta_n)$ are decreasing, and $(\beta_n)$ is increasing. 

Let $L\in \Ns$ be such that 
\begin{equation}\label{ex3-def-L}
L \geq \max \lbr \norm{x_1 - x_0}, \Kp (3 - \rho) + \frac{(1-\rho) \norm{r^*}}{2} \frac{2P+1}{P(P+1)} \rbr.
\end{equation}

\bprop \label{SS-applied-to-ex3}
For all $n \in \N$,
\begin{align}
\norm{x_{n+1} - x_n} &\leq \frac{JL}{(1-\rho)(n+J)}, \label{xnp-xn-inq-SS} \\
\norm{Tx_n - x_n} &\leq \frac{(J+2)L}{(1-\rho)(n+J)}. \label{Txn-xn-inq-SS}
\end{align}
Thus $(x_n)$ is asymptotically regular with rate $\Phi(k)=\ceil{\frac{JL}{1-\rho}(k+1)}$ and 
$T$-asymptotically regular with rate $\Psi(k)=\ceil{\frac{(2+J)L}{1-\rho}(k+1)}$.
\eprop
\begin{proof}
By \eqref{xnpuxn-ineq-main}, we have that
\begin{align}\label{ineq1-ex3}
\norm{x_{n+2} - x_{n+1}} & \leq (1-(1-\rho)\delta_{n+1})\norm{x_{n+1} - x_n} + \Kp d_n 
 + \norm{r_{n+1}-r_n},
\end{align}
where $d_n =  |\delta_{n+1}-\delta_n| +  |\alpha_{n+1}-\alpha_n| + |\beta_{n+1}-\beta_n|$.
As $\alpha_n-\alpha_{n+1}  = \frac{1}{(n+J)(n+J+1)}$, $\beta_{n+1} - \beta_n  =  \frac{3 - \rho}{1 - \rho}(\alpha_n - \alpha_{n+1})$, and 
$\delta_n-\delta_{n+1}  = \frac{2}{1-\rho}(\alpha_n - \alpha_{n+1})$, 
we get that 
\[d_n  =  (\delta_n - \delta_{n+1}) + (\alpha_n - \alpha_{n+1}) + (\beta_{n+1} - \beta_n)
= (3-\rho)(\delta_n - \delta_{n+1}).\]
One can easily see that 
\begin{align*}
\norm{r_{n+1} - r_n} & = 
(\delta_n - \delta_{n+1}) \frac{(1-\rho)\norm{r^*}(n+J)(n+J+1)(2n + 2P + 1)}{2(n + P)^2(n + P+1)^2}.
\end{align*}
Using \eqref{ineq1-ex3}, we obtain
\begin{align}\label{ineq2-ex3}
\norm{x_{n+2} - x_{n+1}} & \leq (1-(1-\rho)\delta_{n+1})\norm{x_{n+1} - x_n} + (\delta_n - \delta_{n+1}) \theta_n,
\end{align}
where 
\begin{align*}
\theta_n & = (3-\rho)\Kp + \frac{(1-\rho)\norm{r^*}(n+J)(n+J+1)(2n + 2P + 1)}{2(n + P)^2(n + P+1)^2}.
\end{align*}
Since $P\geq J$, we get that 	
\begin{align*}
\frac{(2n + 2P + 1)(n+J)(n+J+1)}{2(n + P)^2(n + P+1)^2} & \leq \frac{2n + 2P + 1}{2(n + P)(n + P+1)} 
 = \frac{1}{2(n+P+1)} + \frac{1}{2(n+P)} \\
& \leq \frac{1}{2(P+1)} + \frac{1}{2P} = \frac{2P+1}{2P(P+1)}, 
\end{align*}
Thus, $\theta_n \leq (3-\rho)\Kp + \frac{(1-\rho)\norm{r^*}(2P+1)}{2P(P+1)} \leq L$.

One can easily verify that Lemma \ref{lemma-SabSht-version} can be applied with 
\begin{align*}
s_n := \norm{x_{n+1} - x_n}, \quad a_n := \delta_n, \quad c_n:= \theta_n, \quad \gamma: = 1-\rho, 
\quad N:= 2, \quad J,L \text{ as above}.
\end{align*}
It follows that for all $n\in \N$,
\begin{align*}
\norm{x_{n+1} - x_n} \leq \frac{JL}{\gamma(n+J)} = \frac{JL}{(1-\rho)(n+J)}.
\end{align*}
Thus, \eqref{xnp-xn-inq-SS} holds and, as a consequence, $\Phi$ is a rate of convergence to $0$ of 
$\lp \norm{x_{n+1} - x_n} \rp$.
By \eqref{xnTxn-ineq-main} and \eqref{xnp-xn-inq-SS}, we get that for all $n\in \N$,
\begin{align*}
\norm{Tx_n - x_n } & \leq \norm{x_{n+1}-x_n} + 2 \lp 1-\beta_n\rp \Kp + \norm{r_n} 
\leq \frac{JL}{(1-\rho)(n+J)}  + 2 \lp 1-\beta_n\rp \Kp + \norm{r_n}.
\end{align*}
Remark  that
\begin{align*}
2 \lp 1-\beta_n\rp \Kp + \norm{r_n} & =  2\Kp \frac{3-\rho}{(1-\rho)(n+J)} + \frac{1}{(n+P)^2} \norm{r^*} \\
&  \leq \frac{2}{(1-\rho)(n+J)} \lp \Kp (3-\rho) + \frac{(1-\rho) \norm{r^*}}{2} \frac{1}{n+P} \rp 
\leq  \frac{2L}{(1-\rho)(n+J)}.
\end{align*}
Hence, \eqref{Txn-xn-inq-SS} holds and $\Psi$ is a rate of convergence to $0$ of 
$\lp \norm{Tx_n - x_n} \rp$.
\end{proof}

By letting $\rho=0$ and $f(x)=u\in X$ for all $x\in X$, we obtain the following linear rates of ($T$-)asymptotic regularity  for the 
Kanzow--Shehu iteration \eqref{KanShe-iteration}:
\[ \Phi(k) =JL(k+1), \quad  \Psi(k)=(2+J)L(k+1),\]
where $L\in \Ns$ satisfies $L \geq \max \lbr \norm{x_1 - x_0}, 3\Kp + \frac{\norm{r^*}(2P+1)}{2P(P+1)} \rbr$ and $\Kp$ is 
given by \eqref{def-Kp} with 
\[\Kpz=\max \lbr \norm{x_0 - p}, \norm{u - p}, \norm{p} \rbr.\]
If, moreover, $r^*=0$, then $L$ can be further simplified: $L \geq \max \lbr \norm{x_1 - x_0}, 3\Kp\rbr$.

\end{document}